\documentclass[12pt,leqno]{amsart}

\theoremstyle{thmit} 
\newtheorem{thm}{Theorem}[section]
\newtheorem{lem}[thm]{Lemma}
\newtheorem{cor}[thm]{Corollary}

\newtheorem{definition}[thm]{Definition}

\theoremstyle{thmrm} 

\newtheorem*{rem}{Remark}

\usepackage{amsmath}
\usepackage{mathptmx}
\usepackage{amsfonts}
\usepackage{amssymb}
\usepackage{graphicx}%
\usepackage{color}

\title[Curvature estimates for submanifolds immersed into horocylinders]{Curvature estimates for submanifolds immersed into horoballs and horocylinders}
\author{G. Pacelli  Bessa}
\address{Departamento de Matem\'atica\\Universidade Federal do Cear\'a-UFC\\60455-760 Fortaleza, CE, Brazil.}
\author{Jorge H. de Lira}
\address{Departamento de Matem\'atica\\Universidade Federal do Cear\'a-UFC\\60455-760 Fortaleza, CE, Brazil.}
\author{Stefano Pigola}
\address{Sezione di Matematica - DiSAT\\
Universit\'a dell'Insubria - Como\\
via Valleggio 11\\
I-22100 Como, Italy}
\author{Alberto G. Setti}
\address{Sezione di Matematica - DiSAT\\
Universit\'a dell'Insubria - Como\\
via Valleggio 11\\
I-22100 Como, ITALY}



\thanks{The authors are grateful to the anonymous referee for a very careful reading of the paper. The third and the fourth authors acknowledge partial support by the Italian Miur, Prin 2010-2011 Variet{\`a} reali e complesse: Geometria, topologia e Analisi Armonica, unità locale di Milano-Bicocca, and by INdAM-GNAMPA. They also thank the Departamento de Matem{\' a}tica da Universidade Federal do Cear{\'a}, where this work was carried out, for its warm hospitality. The forth author is grateful to CAPES-Ci{\^ e}ncia Sem Fronteiras, grant 400794/2012-8, for financial support.}

\begin{document}

\maketitle

\begin{abstract}
We prove mean and sectional curvature estimates  for submanifolds confined into either a
horocylinder of $N\times L$ or a horoball of $N$, where $N$ is a Cartan-Hadamard manifold with
pinched curvature. Thus, these submanifolds behave in many respects like submanifolds immersed into
 compact balls and into cylinders over compact balls. The proofs rely on the Hessian comparison theorem for the Busemann function.
\end{abstract}

\section{Introduction}\label{s:1} A classical problem in Riemanninan geometry is to obtain curvature estimates for submanifolds under
extrinsic constraints. Jorge and  Xavier \cite{jorge-xavier},   showed that any complete $m$-dimensional Riemannian manifold $M$
with scalar curvature bounded below isometrically immersed into a normal ball $B_{N}(R)$
of a Riemannian manifold $N$ of  radius $R$ has mean curvature of $M$ satisfying
\begin{equation}
\sup_{M} \vert {\bf H }\vert
\geq \left\{\begin{array}{cll}
\sqrt{ A}\coth(\sqrt{A}R)& {\rm if}& K_{N}\leq -A\,\,{\rm on}\,\,B_{N}(R),\,\,A>0\\
&&\\
\displaystyle\frac{1}{R}& {\rm if}& K_{N}\leq 0\,\,{\rm on}\,\,B_{N}(R)\\
&&\\
\sqrt{A}\cot(\sqrt{ A}R)& {\rm if}& K_{N}\leq A\,\,{\rm on}\,\,B_{N}(R)\,\,{\rm and}\,\,AR< \pi/2,\,\,A>0.
 \end{array}\right.
\end{equation}This result was then generalized relaxing  the scalar curvature condition \cite{Karp-MathAnn}, and ultimately shown to hold provided $M$ is stochastically complete, \cite{PRS-PAMS}, \cite{PRS-Memoirs}. As a result, there are no minimal bounded immersions of a stochastically complete manifold into a Cartan-Hadamard manifold.

In another direction the result was extended even to the case of immersions into cylinders $B_N(R)\times \mathbb{R}^{\ell}$, $\ell < m$,  in a product manifold \cite{ABD-MathAnn}.

In our first theorem we prove mean curvature estimates valid for immersions of a stochastically complete manifold into (suitable) cylinders over unbounded bases. More precisely, let $N$ be an $n$-dimensional Cartan-Hadamard manifold  and let $\sigma$ be  a ray in $N$, with associated Busemann function $b_\sigma$.
The horoball in $N$ determined by $\sigma$ is the set
$$
\mathcal{B}_{\sigma,R}^{n}=\left\{  b_{\sigma}\leq R\right\}, \quad R>0,
$$
and, if $L$ is an arbitrary $\ell$-dimensional manifold, we say that the region
\[
\mathcal{B}_{\sigma,R}^{n,\ell}=\left\{  b_{\sigma}\leq R\right\}  \times L
\]
is a (generalized solid) horocylinder in $N\times L$. With this notation we have
\begin{thm}
\label{prop_submanifold}Let $(  N,g_{N})  $ be an $n$-dimensional
Cartan-Hadamard manifold with sectional curvature satisfying $-B\leq\mathrm{Sec}_{N}\leq-A$ and let
$\sigma:[0,+\infty)\rightarrow N$ be a ray of $N$. Let $(  L,g_{L}%
)  $ be any $\ell$-dimensional Riemannian manifold and let $f=(
f_{N},f_{L})  :\Sigma\hookrightarrow N\times L$ be an $m$-dimensional
isometric immersion with mean curvature vector field $\mathbf{H}$. Assume that
$m>\ell$. If $\Sigma$ is stochastically complete and $f(  \Sigma)  $
is contained in the horocylinder $\mathcal{B}_{\sigma,R}^{n,\ell}$, then%
\[
\sup_{\Sigma}\left\vert \mathbf{H}\right\vert \geq\frac{m-\ell}{m}\sqrt A.
\]
\end{thm}

We explicitly note that the result remains true when the fibre $L$ degenerates to a $0$-dimensional point and, hence,
the horocylinder reduces to a horoball.

Since bounded mean curvature submanifolds properly immersed into a complete ambient manifold of
bounded sectional curvature are stochastically complete, see for instance
\cite{PRS-Memoirs}, as a corollary we have

\begin{thm} \label{th_meancurv}
Let $(  N,g_{N})  $ be an $n$-dimensional Cartan-Hadamard manifold with sectional curvature
satisfying $-B\leq\mathrm{Sec}_{N}\leq-A<0$ and let $(  L,\,g_{L})  $
be any complete $\ell$-dimensional Riemannian manifold with sectional curvature $\mathrm{Sec}_{L}\geq -C^{2}$, where $A,B,C$ are positive constants. If
$f\colon \Sigma\hookrightarrow N\times L$ is an $m$-dimensional properly immersed
submanifold with $f(\Sigma)$ inside a horocylinder of $N\times L$ and  $m>\ell$ then the mean
curvature vector $\mathbf{H}$ of the immersion satisfies%
\[
\sup_{\Sigma}\left\vert \mathbf{H}\right\vert \geq\frac{m-\ell}{m}\,\sqrt A.
\]
\end{thm}

Our second result is the following sectional curvature lower estimate in the spirit of the classical theorem by Jorge-Koutroufiotis, \cite{JK-AJM}.
We point out that, although it is stated for submanifolds in a horoball, one can prove a version for submanifolds contained in a
horocylinder of $N\times L$, where $\mathrm{Sec}_L\geq -B$, under suitably modified assumptions on the dimensions.

\begin{thm}\label{th_JK}
Let $(  N,g_{N})  $ be an $n$-dimensional Cartan-Hadamard manifold with sectional curvature
satisfying $-B\leq\mathrm{Sec}_{N}\leq-A<0$  where $A,B$ are positive constants. Let
$f\colon\Sigma\hookrightarrow N$ be an $m$-dimensional
 submanifold  properly immersed with $f(M)$ inside a horoball of $N$ and $n\leq 2m-1$ then the sectional curvature of $\Sigma$ satisfies the estimate%
\[
\sup_{\Sigma}\mathrm{Sec}_\Sigma \geq A-B.
\]
\end{thm}

As the geometric setting suggests, the main tool to obtain the results is the
analysis of the Busemann function in Cartan-Hadamard manifolds which will be described in the next Section.

\section{Busemann functions in Cartan-Hadamard spaces}

Throughout this section we let $(  N,g_{N})  $ be a Cartan-Hadamard
manifold, i.e.,  a  simply connected, complete Riemannian manifold of
non-positive sectional curvature. Further assumptions on $N$ will be
introduced when required. First, we are going to collect some of the basic
differentiable properties of the Busemann function of $N$ with respect to a
fixed geodesic ray. Since we are not aware of any specific reference we
decided to provide fairly detailed proofs.

Let $\sigma:[0,+\infty)\rightarrow N$ be a geodesic ray issuing from
$\sigma(  0)  =o$. Recall that, by its very definition, the
Busemann function of $N$ \ with respect to $\sigma$ is the function
$b_{\sigma}:N\rightarrow\mathbb{R}$ defined by
\[
b_{\sigma}(  x)  =\lim_{t\rightarrow+\infty}b_{\sigma(t)}(
x)
\]
where, for any fixed $t\geq0$,
\[
b_{\sigma(  t)  }(  x)  =r_{\sigma(  t)
}(  x)  -r_{\sigma(  t)  }(  o)
=r_{\sigma(  t)  }(  x)  -t.
\]
Here and below, $r_{p}(  x)  =d(  p,x)  $ denotes the
distance function from a point $p$. In some sense, the Busemann function
measures the distance of the points of $N$ from the point $\sigma(+\infty)  $ in the ideal boundary $\partial N$. Since $t\longmapsto
b_{\sigma(  t)  }(  x)  $ is monotone decreasing and
bounded by $\left\vert b_{\sigma(  t)  }(  x)
\right\vert \leq r_{o}(  x)  $, the limit $b_{\sigma}(
x)  $ exists and is finite. Moreover, the convergence is uniform on
compacts by Dini's theorem. Clearly, by the triangle inequality, each $b_{\sigma(
t)  }$ is $1$-Lipschitz (in fact, $|\nabla b_{\sigma(t))}| =1$ by the Gauss Lemma) and, therefore, so is also $b_{\sigma}$. In
particular, $b_{\sigma}$ is differentiable a.e. Actually, in the special case
of Cartan-Hadamard manifolds it was proved by P. Eberlein, \cite{HH-JDG}, that
$b_{\sigma}$ is a function of class $C^{2}$. To begin with we observe that the
gradient $\nabla b_{\sigma}$ of the Busemann function can be obtained via a
limit procedure from $\nabla b_{\sigma(  t)  }$ as $t\rightarrow
+\infty$.

\begin{lem}
[limit of gradients]\label{lemma_gradient}Assume that the sectional curvature
of $N$ is bounded, namely, there exists $B>0$ such that $-B\leq\mathrm{Sec}%
_{N}\leq0$. Then, for every $x\in N$,%
\[
\lim_{t\rightarrow+\infty}\nabla b_{\sigma(  t)  }(  x)
=\nabla b_{\sigma}(  x)
\]
and the convergence is locally uniform.
\end{lem}

\begin{proof}
By the Hessian comparison theorem, \cite{Pe-book}, we know that, having fixed a compact ball
$K\subset N$, \ there exist $T=T(  K)  >0$ and $C=C(
K,B)  >0$ \ such that%
\[
\left\vert \mathrm{Hess}\,b_{\sigma(  t)  }\right\vert \leq C,
\]
for every $x\in K$ and for every $t\geq T$. It follows that for any sequence
$\left\{  t_{k}\right\}  \rightarrow+\infty$ the corresponding sequence of
gradients $\left\{  \nabla b_{\sigma(  t_{k})  }\right\}  $ is
eventually equi-continuous on $K$. Since it is equi-bounded as observed above,
we deduce that there exists a subsequence $\{\nabla b_{\sigma(t_{k_{j}})}\}$
that converges uniformly on $K$ to a continuous vector field $\xi$ on $K$. On
the other hand, the sequence $\{\nabla b_{\sigma (t_{k})}\}$ converges weakly
to $\nabla b_{\sigma}$ on compact sets. Indeed, if $V$ is a smooth vector
field supported in a ball $B_{R}$, then, by dominated convergence,%
\[
\int_{B_{R}} \left\langle \nabla b_{\sigma(  t_{k})  },V\right\rangle
=-\int_{B_{R}}b_{\sigma(  t_{k})  }\operatorname{div}%
V\rightarrow-\int_{B_{R}}b_{\sigma}\operatorname{div}V=\int_{B_{R}}\left\langle \nabla
b_{\sigma},V\right\rangle ,
\]
as claimed. It follows that%
\[
\xi=\nabla b_{\sigma}%
\]
a.e. on $K$ and, in fact, everywhere on $K$ by continuity. Since the selected
sequence $\left\{  t_{k}\right\}  $ was arbitrary, the required conclusion follows.
\end{proof}

In the above proof the conclusion is obtained  using the weak definition of
the gradient, which behaves well under limits, together with the fact the weak
gradient agrees with the classical gradient for sufficiently regular
functions. A similar trick will be used in the next result where we deduce a
comparison principle for the Hessian of the Busemann function. We put the following:
\begin{definition}
A function $h:N\rightarrow\mathbb{R}$ is said to satisfy the differential
inequality
\[
\mathrm{Hess}\,h\leq \mathcal{T}
\]
in the sense of distributions, where $\mathcal{T}$ is a symmetric $2$-tensor,
if the integral inequality
\[
\int_{N}h\left\{  \operatorname{div}(  V\operatorname{div}V)
+ \operatorname{div}D_{V}V\right\}  \leq\int_{N}\mathcal{T}(  V,V)
\]
holds for every smooth compactly supported vector field $V$.
\end{definition}

Clearly, in case $h$ is of class $C^{2}$, a double integration by parts shows that the
distributional inequality is equivalent to
\[
\int_{N}\mathrm{Hess}\,h(  V,V)  \leq\int_{N}\mathcal{T}(  V,V).
\]
The validity of this latter for every compactly supported vector field $V$, in turn, is equivalent to the pointwise inequality. Indeed, suppose
\[
\mathrm{Hess}_{x}\,h(v,v) > \mathcal{T}_{x}(v,v),
\]
for some $x\in M$ and some $v\in T_{x}M\setminus\{0\}$. Extend $v$ to any smooth vector field $V'$ on $M$. By continuity, there exists a neighborhood $\mathcal{U}$ of $x$ such that,
\[
\mathrm{Hess}\,h (V',V') > \mathcal{T}(V',V'), \;\text{ on } \mathcal{U}.
\]
To conclude, choose a smooth cut-off function $\xi: M \to \mathbb{R}$ such that $\mathrm{supp}\,\xi \subset \mathcal{U}$ and $\xi =1$ at $x$, and observe that $V = \xi V'$ violates the distributional inequality.

\begin{lem}
[Hessian comparison]\label{lemma_hessian}Assume that the sectional curvatures
of $N$ satisfy $$-B\leq\mathrm{Sec}_{N}\leq-A$$ \ for some constants $B\geq A>0$.
Then%
\[
\sqrt{A}(  g_{N}-db_{\sigma}\otimes db_{\sigma})  \leq
\mathrm{Hess}b_{\sigma}\leq\sqrt{B}(  g_{N}-db_{\sigma}\otimes
db_{\sigma})  ,
\]
in the sense of quadratic forms.
\end{lem}

\begin{proof}
Let us show how to prove the upper bound of $\mathrm{Hess}\,b_{\sigma}$.
Obviously, the lower bound can be obtained using exactly the same arguments.
By the Hessian comparison theorem, having fixed a ball $B_{R}$ of $N$, we find
$T>0$ such that, for every $t\geq T$,%
\[
\mathrm{Hess}\,b_{\sigma(  t)  }\leq \sqrt{B}\coth(
r_{\sigma(  t)  }\sqrt{B})  \left\{  (  g_{N}%
-db_{\sigma(  t)  }\otimes db_{\sigma(  t)  })
\right\}  ,\text{ on }B_{R}.
\]
In particular, this inequality holds in the sense of distribution, namely, for
every vector field $V$ compactly supported in $B_{R},$ it holds%
\[
\int_{N}b_{\sigma(  t)  }\left\{  \operatorname{div}(
V\operatorname{div}V)  + \operatorname{div}D_{V}V\right\}  \leq\int
_{N}\left\{  \left\vert V\right\vert ^{2}-\left\langle \nabla b_{\sigma(
t)  },V\right\rangle ^{2}\right\}  .
\]
Evaluating this latter along a sequence $\left\{  t_{k}\right\}
\rightarrow+\infty$, using Lemma \ref{lemma_gradient}, and applying the
dominated convergence theorem we deduce that the integral inequality%
\[
\int_{N}b_{\sigma}\left\{  \operatorname{div}(  V\operatorname{div}%
V)  + \operatorname{div}D_{V}V\right\}  \leq\int_{N}\left\{  \left\vert
V\right\vert ^{2}-\left\langle \nabla b_{\sigma},V\right\rangle ^{2}\right\}
\]
holds for every  smooth vector field compactly supported in $B_{R}$. To conclude we now
recall that this is equivalent to the pointwise inequality%
\[
\mathrm{Hess}\,b_{\sigma}\leq\sqrt{B}(  g_{N}-db_{\sigma}\otimes
db_{\sigma})  ,
\]
on $B_{R}$.
\end{proof}
We remark that a version of the above lemma was also observed without proof in
\cite{CM}.

\begin{cor}
\label{cor_eb}Keeping the notation and the assumptions of Lemma
\ref{lemma_hessian}, let $u:N\rightarrow\mathbb{R}$ be the smooth function
defined by%
\[
u(  x)  =e^{\sqrt{ A}\; b_{\sigma}(  x)  }.
\]
Then%
\[
A\, u\cdot g_{N}\leq\mathrm{Hess}\, u  \leq\sqrt{AB}\,u\cdot
g_{N}.
\]
\end{cor}

\section{Proof of the main theorems}
We are now ready to give the proof of our results.
\begin{proof} [Proof of Theorem~\ref{prop_submanifold}]
Let
\[
w=u\circ f_{N}:\Sigma\rightarrow\mathbb{R}_{>0}%
\]
where $u:N\rightarrow\mathbb{R}$ is the smooth function defined in Corollary
\ref{cor_eb}. By the composition law for the Laplacians we have
\[
\Delta w=\mathrm{tr}_{\Sigma}\left\{  \mathrm{Hess} \, u  (
df_{N}\otimes df_{N})  \right\}  +du(  \mathrm{tr}_{\Sigma}%
Ddf_{N})  .
\]
On the other hand, by Corollary \ref{cor_eb}, 
\[
\mathrm{Hess} \, u  \geq A\, u\cdot g_{N}%
\]
so that
\begin{align*}
\Delta w  &  \geq w\left\{ A\, \mathrm{tr}_{\Sigma}g_{N}(  df_{N}\otimes
df_{N})  -m\sqrt{A}\left\vert \mathbf{H}\right\vert \right\} \\
&  =w\left\{ A \,\mathrm{tr}_{\Sigma}f_{N}^{\ast}(  g_{N})
-m\sqrt{A}\left\vert \mathbf{H}\right\vert \right\}  .
\end{align*}
Since $f^{\ast}g_{N\times L}=g_{\Sigma}$ then%
\begin{align*}
\mathrm{tr}_{\Sigma}f_{N}^{\ast}(  g_{N})   &  =m-\mathrm{tr}%
_{\Sigma}f_{L}^{\ast}(  g_{L}) \\
&  \geq m-\ell,
\end{align*}
and from the above we conclude that%
\begin{equation}
\Delta w(  x)  \geq mw(  x)  \left\{  A \frac{m-\ell}{m}%
-\sqrt{A}\sup_{\Sigma}\left\vert \mathbf{H}(  x)  \right\vert \right\}  .
\label{submanifold-laplacian}%
\end{equation}
Now we apply the weak maximum principle for the Laplacian, \cite{PRS-PAMS}, to
get%
\[
0\geq m\sup_{\Sigma}w\left\{  A\frac{m-\ell}{m} -\sqrt{A}\sup_{\Sigma}\left\vert
\mathbf{H}(  x)  \right\vert \right\}  ,
\]
as required.
\end{proof}

\begin{rem}
{\rm
By applying the strong maximum principle to (\ref{submanifold-laplacian}) we
can also obtain directly the following touching principle.\medskip

\noindent
Let $f:\Sigma\hookrightarrow N\times L$ be a complete, immersed
submanifold, where $N$ is a Cartan-Hadamard manifold of pinched negative
curvature and $L$ is any complete Riemannian manifold. Let us assume that (a) the
mean curvature $\mathbf{H}$ of the immersion satisfies $\left\vert
\mathbf{H}\right\vert \leq \sqrt{A}(  m-\ell)  /m$ and that (b) $f(
\Sigma)  $ is contained inside the horocylinder $\mathcal{B}_{\sigma
,R}^{n,\ell}$ and $f(  \Sigma)  \cap\partial\mathcal{B}_{\sigma
,R}^{n,\ell}\neq0.$ Then $f(  \Sigma)  =\partial\mathcal{B}_{\sigma
,R}^{n,\ell}.$
}
\end{rem}

\begin{rem}
{\rm
In a different direction, if $f=f_{N}:\Sigma\rightarrow\mathcal{B}_{\sigma
,R}^{n,0}=\left\{  b_{\sigma}\leq R\right\}\subset N$ is a proper immersion into a horoball of $N$ and
\[
\sup_{\Sigma}\left\vert \mathbf{H}(  x)  \right\vert <\sqrt{A}
\]
then $w$ is a bounded exhaustion function that violates the weak maximum
principle at infinity. By Theorem 32 in \cite{BPS-Revista} it follows that the
essential spectrum of $\Sigma$ is empty.
}
\end{rem}
The estimates for the Hessian of the Busemann function
allows us to obtain also the Jorge-Koutroufiotis type result stated in Theorem \ref{th_JK}.
This results give a further indication of the phenomenon according to which submanifolds of non-compact horoballs
behave in many respects like a submanifolds of compact balls.

\begin{proof}[Proof (of Theorem \ref{th_JK})]
The proof is similar to the arguments in \cite{BLP-AMPA}.
For every $k$ consider the function $h_k:\Sigma \to \mathbb{R}$
$$
h_k = w - \frac 1k [\log( \rho_N \circ f +1) +1 ],
$$
where, as above $w=e^{\sqrt{ A}\, b_\sigma}\circ f$ and $\rho_N$ denotes the Riemannian distance in $N$ from an origin $o$ in the complement of $\mathcal{B}_{\sigma,R}^{n,0}$.
Since $f(\Sigma)$ is contained in a horoball, the first summand is bounded above, and since the $f$ is proper in $N$, the second summand tends to infinity at infinity. It follows that for every $k$, $h_k$ attains an absolute maximum at a point $x_k$ where
$$
\mathrm{Hess}\, h_k = \mathrm{Hess}\, w - \frac 1k \mathrm{Hess} [\log( \rho_N \circ f +1) +1 ] \leq 0
$$
in the sense of quadratic forms.  Now,  according to our previous computations, for all  vectors $X_k\in T_{x_k}\Sigma$,
\begin{equation}
\label{JK1}
\mathrm{Hess}\,w (X_k,X_k)
\geq \sqrt A \,w(x_k)(\sqrt A |X_k|^2 - |\mathrm {II}(X_k,X_k)|),
\end{equation}
where $II$ is the second fundamental form of the immersion.
On the other hand, by the Hessian comparison theorem,
$$
\mathrm{Hess} \,\rho_N \leq \sqrt B \coth(\sqrt B \,\rho_N) (g_N - d\rho_N\otimes d\rho_N),
$$
and after some computation we obtain
\begin{multline*}
\mathrm{Hess}\,(
[\log( \rho_N\circ f +1) +1 ]
)
 (X_k,X_k)
\\
\leq \frac 1 {\rho_N(f(x_k))+1}
\left\{
\sqrt B \coth(\sqrt B \rho_N(f(x_k)))|X_k|^2
+ |\mathrm{II}(X_k,X_k)|
\right\}
\end{multline*}
Combining the two inequalities and rearranging we conclude that
\begin{multline*}
|\mathrm{II}(X_k,X_k)|
\biggl( \sqrt A w(x_k) + \frac 1 {k[\rho_N(f(x_k)) +1]}
\biggr)
\\
\geq
\biggl(
Aw(x_k) - \frac{\sqrt{B} \coth(\sqrt{B}\rho_N(f(x_k)) )}{k[\rho_N(f(x_k))+1]}
\biggr)
 |X_k|^2
\end{multline*}
Now notice that $w(x_k)$ is positive and bounded away from $0$. Indeed, if $\bar x$ is a point such that
$\rho_N(f(\bar x))= \min_\Sigma\rho_N(f(x)),$ then for every $k$ we have
$$
h_k(x_k) \geq h_k(\bar x)
$$
and therefore
$$
w(x_k) \geq w(\bar x) + \frac 1k \{ \log[\rho_N (f(x_k))+1] -  \log[\rho_N (f(\bar x))+1]\} \geq w(\bar x).
$$
Since $\rho_N(f(x_k))$ is also bounded away from zero, it follows that for every sufficiently large $k$, and every non zero $X_k\in T_{x_k}\Sigma$,
$$\begin{array}{lll}
|\mathrm{II}(X_k,X_k)|&\geq & {\displaystyle
\frac
{
Aw(x_k) - \displaystyle\frac{\sqrt{B} \coth(\sqrt{B}\,\rho_N(f(x_k)) )}{k[\rho_N(f(x_k))+1]}
}
{
\sqrt A w(x_k) +  \displaystyle\frac 1{k [\rho_N(f(x_k)) +1]}}|X_k|^2
} \\
&&\\
& =& [\sqrt A+ o(k^{-1})] |X_k|^2\end{array}
$$
In particular, $\mathrm{II} (X_k,X_k)>0$ for every sufficiently large $k$ and every $X_k\in T_{x_k}\Sigma\setminus\{0\}$
and we may apply Otsuki's lemma (see, e.g., \cite{KN-bookII}, p. 28) to find unit tangent vectors $X_k$ and $Y_k$ such that
$\mathrm{II} (X_k,X_k)=\mathrm{II}(Y_k,Y_k)$ and $\mathrm{II}(X_k,Y_k)=0$. The required conclusion now follows from Gauss equations as in the original proof:
\begin{equation*}
\begin{split}
\mathrm{Sec}_\Sigma (X_k,Y_k)&= \mathrm{Sec}_N (df X_k, df Y_k)  +  g_N (\mathrm{II}(X_k,X_k), \mathrm{II}(Y_k,Y_k))
- |\mathrm{II}(X_k,Y_k)|^2 \\&
\geq -B + A + o(k^{-1}).
\end{split}
\end{equation*}
\end{proof}
Again as in the classical proof, we note that the conclusion of the theorem follows directly from \eqref{JK1} if we assume
that the weak maximum principle for the Hessian holds, for then there exists a sequence $x_k$ such that $w(x_k)\to \sup_\Sigma w$ and
$$
\mathrm{Hess}\, w (x_k) \leq 1/k g_\Sigma,
$$
which together with \eqref{JK1} allows to conclude as in the last part of the above proof. In
 particular, the conclusion holds if $\mathrm{Scal}_\Sigma\geq - G(\rho_\Sigma)$  where the function $G$  is positive, non-decreasing and $G^{-1/2}$ is integrable at infinity. Indeed, assuming that $\mathrm{Sec}_\Sigma$ is bounded above, for otherwise the conclusion holds trivially, then $\mathrm{Sec}_\Sigma$ is bounded below by a multiple of $-G$ and the Omori-Yau maximum principle for the Hessian holds on $\Sigma$ (\cite{PRS-Memoirs}).

\bibliographystyle{line}

\end{document}